\documentclass[11pt]{amsart}
\usepackage[OT1]{fontenc}

\usepackage{amsmath,amssymb,amsthm,mathtools}
\usepackage{geometry}
\usepackage[american]{babel}
\usepackage{setspace}
\usepackage[dvipsnames]{xcolor}
\usepackage{tikz}
\usepackage{hyperref}
\hypersetup{colorlinks=true,linktocpage=true,citecolor=blue,linkcolor=blue,urlcolor=blue}
\hypersetup{
  pdftitle={Spectral Geroch conjecture and noncompact area enlargeable summands},
  pdfauthor={Daoqiang Liu},
  pdfkeywords={scalar curvature, area enlargeability, generalized Geroch conjecture, spectral constant, spin covering}
}

\theoremstyle{plain}
\newtheorem{theorem}{Theorem}[section]
\newtheorem{mainthm}{Theorem}

\newtheorem{maincor}[mainthm]{Corollary}
\newtheorem{lemma}[theorem]{Lemma}
\newtheorem{proposition}[theorem]{Proposition}
\newtheorem{corollary}[theorem]{Corollary}

\theoremstyle{definition}
\newtheorem{definition}[theorem]{Definition}
\newtheorem{remark}[theorem]{Remark}

\numberwithin{equation}{section}

\DeclareMathOperator{\scal}{scal}
\DeclareMathOperator{\supp}{supp}
\newcommand{\Sph}{\mathbf{S}}
\newcommand{\T}{\mathbf{T}}
\newcommand{\Ahat}{\mathbf{\widehat{A}}}

\begin{document}
\raggedbottom

\title[Spectral Geroch conjecture and noncompact enlargeable summands]{Spectral Geroch conjecture and noncompact area enlargeable summands}

\author{Daoqiang Liu}
\address{Chern Institute of Mathematics \& LPMC, Nankai University, Tianjin 300071, China}
\email{\href{mailto:dqliu@nankai.edu.cn}{dqliu@nankai.edu.cn}}
\urladdr{\href{https://www.dqliu.cn}{www.dqliu.cn}}

\subjclass[2020]{53C21, 53C27, 53C23, 58J30}
\keywords{scalar curvature, area enlargeability, generalized Geroch conjecture, spectral constant, spin covering}

\begin{abstract}
We study complete metrics on area enlargeable manifolds through the spectral constant $\Lambda_\gamma(g)$, the infimum of the Rayleigh quotient of $-\Delta+\gamma\scal_g$ over compactly supported functions. We prove that every complete metric on an area enlargeable $m$-manifold, $m\ge2$, satisfies $\Lambda_\gamma(g)\le0$ for all $\gamma>(m-1)/(4m)$. For an area enlargeable manifold $M_1$ and an oriented manifold $M_2$ of the same dimension, we prove that $M_1\#M_2$ is area enlargeable if and only if $M_2$ admits a spin covering. Both summands may be noncompact. Consequently, if the universal covering of $M_2$ is spin, then $M_1\#M_2$ carries neither a complete metric of positive scalar curvature nor a complete metric with positive $\gamma$-spectral constant for $\gamma>(m-1)/(4m)$. The threshold $(m-1)/(4m)$ is optimal. At the endpoint, the obstruction holds for metrics whose scalar curvature is bounded from below, but explicit complete warped metrics on $\T^{m-1}\times\mathbf R$ show that it fails in general. We also prove that the locally constant and the classical formulations of area enlargeability coincide in dimensions at least two.
\end{abstract}

\maketitle

\section{Introduction}

Area enlargeability, introduced by Gromov and Lawson \cite{GL80,GL83}, obstructs complete metrics of positive scalar curvature on noncompact manifolds \cite[Theorem~6.12]{GL83}. The classical definition requires that, for every $\varepsilon>0$, there be a spin covering carrying a map to the round sphere that has nonzero degree, is area $\varepsilon$-decreasing, and is constant outside a compact set. Shi \cite[Definition~1.10]{Shi25} and Su \cite[Definition~1.2]{Su26} allowed these maps to be locally constant outside a compact set. Such maps arise in constructions of connected sums with a noncompact summand. Su proved that this formulation still obstructs complete metrics of positive scalar curvature \cite[Theorem~1.4]{Su26}. We use this definition (Definition~\ref{def:enlargeable}) and prove in Lemma~\ref{lem:normalize-infinity} that it gives the same area enlargeable metrics and manifolds as the classical definition in dimensions at least two.

Closed area enlargeable manifolds include tori, closed manifolds of nonpositive sectional curvature, and closed spin manifolds admitting a map of nonzero degree to an enlargeable manifold \cite{GL80,GL83}. Noncompact examples include the products $W\times\mathbf R$ with $W$ closed and enlargeable \cite[Proposition~6.7]{GL83} and the complete hyperbolic manifolds of finite volume \cite[Example~6.11]{GL83}.

The Geroch conjecture asserts that the torus $\T^m$ carries no metric of positive scalar curvature. It was proved by Schoen and Yau \cite{SY79} for $3\le m\le7$ and by Gromov and Lawson \cite{GL80,GL83} in all dimensions. The \emph{generalized Geroch conjecture} asserts that, for every $m$-manifold $X$, compact or not, the connected sum $\T^m\#X$ admits no complete metric of positive scalar curvature. Chodosh and Li \cite{CL24} proved it for $3\le m\le7$. The case $m=3$ is also due to Lesourd, Unger and Yau \cite{LUY24}. Chen \cite[Theorem~1.3]{Che24} proved the same obstruction for $W\#X$, where $W$ is a closed Schoen--Yau--Schick $m$-manifold, $3\le m\le7$, and $X$ is an arbitrary $m$-manifold, without a spin assumption. For closed $X$ and $3\le m\le 12$, Bi, Hao, He, Shi and Zhu \cite[Theorem~1.9]{BHHSZ26}, building on the generic regularity theory of Chodosh, Mantoulidis, Schulze and Wang \cite{CMSW25}, proved that $\T^m\#X$ carries no metric of positive scalar curvature and that every metric of nonnegative scalar curvature on it is flat. In the spin setting, Wang and Zhang \cite{WZ22} proved the following.

\begin{theorem}\label{thm:WZ}
Let $W$ be a closed area enlargeable manifold and $X$ a spin manifold of the same dimension. Then $W\#X$ admits no complete metric of positive scalar curvature.
\end{theorem}

For noncompact $X$, the proof combines deformed Dirac operators with a covering construction. Compactness of $W$ supplies a uniform positive lower bound for the scalar curvature on the complement in $W$ of the ball used to form the connected sum, where Zhang's noncompact Llarull theorem \cite{Zha20} is applied. Su \cite[Proposition~1.3]{Su26} proved that $W\#X$ is area enlargeable in the locally constant sense when $W$ is closed and $X$ is noncompact and spin. Together with \cite[Theorem~1.4]{Su26}, this gives another proof of Theorem~\ref{thm:WZ}.

Spectral versions of these obstructions replace pointwise positivity of the scalar curvature by positivity of the bottom of the spectrum of an operator $-\Delta+\gamma\scal$ (see \cite{HKKZ24} for background). Chai and Sun \cite[Theorem~1.21]{CS26} proved with warped $\mu$-bubbles that, for $3\le m\le7$ and every $m$-manifold $X$, the connected sum $\T^m\#X$ admits no complete metric for which $-\beta u^{-1}\Delta u+\frac12\scal$ is positive for some positive function $u$ and some $0\le\beta<\frac{2m}{m-1}$. For $\beta>0$, write $L_\beta=-\beta\Delta+\frac12\scal$ and let $\lambda_0(L_\beta)$ denote its spectral bottom defined using compactly supported test functions. The condition $\lambda_0(L_\beta)>0$ implies this hypothesis. Indeed, on a noncompact complete manifold, choose $0<\varepsilon<\lambda_0(L_\beta)$ and apply the existence theorem for positive solutions \cite[Theorem~1]{FCS80} to $\beta^{-1}(L_\beta-\varepsilon)$ to obtain a smooth function $u>0$ satisfying $L_\beta u=\varepsilon u>0$. On a closed manifold, use a positive eigenfunction for the principal eigenvalue of $L_\beta$. For $\beta=0$, positivity of the spectral bottom means $\inf\scal>0$, and one may take $u=1$. In the normalization used below, the range $0<\beta<\frac{2m}{m-1}$ corresponds to $\gamma=\frac1{2\beta}>\frac{m-1}{4m}$.
Related results include the rigidity theorem of He, Shi and Yu \cite[Corollary~2.7]{HSY25} for manifolds of dimension at most seven admitting a quasi-proper map of nonzero degree to a closed enlargeable manifold, with the mapping assumptions and degree understood as in \cite[Appendix~B]{HSY25}. The spectral parameter is $\gamma\ge\frac12$ in our normalization. Another related result is the weighted scalar curvature obstruction of Bi and Zhu \cite[Theorem~B]{BZ26} for proper, globally Lipschitz maps of nonzero degree to complete length-enlargeable manifolds, under their weighted curvature and parameter assumptions. The present paper works in all dimensions, with spin coverings and Dirac operators.

We use the sign convention $\Delta=\operatorname{div}\nabla$, so that $-\Delta\ge0$. For a Riemannian manifold $(M,g)$ and $\gamma\in\mathbf R$, the \emph{$\gamma$-spectral constant} of $g$ is
\begin{equation}\label{eq:spectral-definition}
\Lambda_\gamma(g)=\inf_{0\ne u\in C_c^\infty(M)}
\frac{\displaystyle\int_M\bigl(|\nabla u|^2+\gamma\scal_g u^2\bigr)\,dV_g}
{\displaystyle\int_M u^2\,dV_g}\in[-\infty,\infty).
\end{equation}
Here and below, scalar test functions are real-valued. Thus $C_c^\infty(M)$ means $C_c^\infty(M,\mathbf R)$. Compactly supported Lipschitz test functions give the same infimum. Indeed, they can be smoothed within a fixed compact neighborhood of the support, on which $\scal_g$ is bounded. No integrability of $\scal_g$ is required. If the quadratic form in \eqref{eq:spectral-definition} is bounded from below on $C_c^\infty(M)$, then $\Lambda_\gamma(g)$ is the bottom of the spectrum of the Friedrichs extension of $-\Delta+\gamma\scal_g$. If $\scal_g$ is bounded, it also equals the corresponding infimum over $W^{1,2}_0(M)$.

In general, $C_c^\infty(M)\subset W^{1,2}_0(M)$, so whenever the Sobolev infimum is defined and positive, $\Lambda_\gamma(g)$ is positive as well. The obstructions below therefore apply to the Sobolev version. We write $\lambda_1(g)=\Lambda_0(g)$ for the bottom of the spectrum of the Laplacian. Since $\Lambda_\gamma(g)\ge\lambda_1(g)+\gamma\inf_M\scal_g$ for $\gamma>0$, a metric with $\scal_g\ge\delta_0>0$ satisfies $\Lambda_\gamma(g)\ge\gamma\delta_0>0$ for every $\gamma>0$.

We first prove a spectral obstruction for area enlargeable manifolds.

\begin{mainthm}\label{thm:A}
Let $M$ be an area enlargeable $m$-manifold, $m\ge2$. Then every complete Riemannian metric $g$ on $M$ satisfies
\[
\Lambda_\gamma(g)\le0\qquad\text{for every }\gamma>\frac{m-1}{4m}.
\]
\end{mainthm}

By the preceding remark, Theorem~\ref{thm:A} contains the obstruction to complete metrics with $\scal_g\ge\delta_0>0$. Complete metrics of pointwise positive scalar curvature on such $M$ are excluded by Su's theorem \cite[Theorem~1.4]{Su26}. In view of Lemma~\ref{lem:normalize-infinity}, this obstruction also follows from \cite[Theorem~6.12]{GL83}. The proof of Theorem~\ref{thm:A} combines monotonicity of $\Lambda_\gamma$ under Riemannian coverings (Proposition~\ref{prop:spectrum-covering}) with a Dirac estimate for sphere maps with uniformly small area contraction constants (Proposition~\ref{prop:spectral-sphere}), in which the coefficient $\frac{m-1}{4m}$ comes from an inequality of Kato type for twisted spinors (Lemma~\ref{lem:dirac-gap}).

To apply Theorem~\ref{thm:A} when both summands may be noncompact, we prove the following criterion for area enlargeability.

\begin{mainthm}\label{thm:B}
Let $M_1$ and $M_2$ be oriented $m$-manifolds, $m\ge2$, and let $M_1$ be area enlargeable. The following are equivalent:
\begin{enumerate}
\item[(i)]\leavevmode\phantomsection\label{thm:B-i}$M_1\#M_2$ is area enlargeable;
\item[(ii)]\leavevmode\phantomsection\label{thm:B-ii}$M_2$ admits a spin covering;
\item[(iii)]\leavevmode\phantomsection\label{thm:B-iii}the universal covering of $M_2$ is spin.
\end{enumerate}
\end{mainthm}

The proof of \hyperref[thm:B-ii]{(ii)}$\Rightarrow$\hyperref[thm:B-i]{(i)} in Section~\ref{sec:B} collapses a neighborhood of the neck on the enlargeable side of the connected sum and glues spin coverings of the two summands in a complete bipartite pattern. The implication \hyperref[thm:B-i]{(i)}$\Rightarrow$\hyperref[thm:B-ii]{(ii)} shows that if no covering of $M_2$ is spin, then $M_1\#M_2$ is not area enlargeable. For example, $\T^4\#\mathbf{CP}^2$ is not area enlargeable, although it carries no metric of positive scalar curvature by the minimal hypersurface method \cite{SY79}. See also \cite[Theorem~3]{CL24}. It is, however, enlargeable in the sense of Cecchini and Schick \cite[Definition~1.1]{CS21}, who do not require the coverings to be spin: pull back large finite coverings of $\T^4$ along the map $\T^4\#\mathbf{CP}^2\to \T^4$ of degree one that collapses the second summand.

Theorem~\ref{thm:B} extends the observation of Gromov and Lawson that the connected sum of an area enlargeable manifold with a compact spin manifold is area enlargeable \cite[\S6, after Proposition~6.5]{GL83}. It also extends Su's result \cite[Proposition~1.3]{Su26} for a closed enlargeable summand and a noncompact spin summand.

When $M_2$ is noncompact, the map $M_1\#M_2\to M_1$ collapsing the second summand is not proper, so the argument of \cite{GL83} does not apply. When $M_1$ is noncompact, the restriction of a metric on $M_1\#M_2$ to the enlargeable side need not be uniformly equivalent to a metric on $M_1$ fixed in advance. Compactness of the enlargeable summand is used at this point in \cite[Proposition~1.3]{Su26}, and to obtain a positive lower bound for the scalar curvature in the proof of Theorem~\ref{thm:WZ}. We instead extend the given metric across the neck and use that area enlargeability of a manifold refers to all of its metrics. We obtain the following obstruction for connected sums.

\begin{maincor}\label{cor:main}
Let $M_1$ and $M_2$ be oriented $m$-manifolds, $m\ge2$. Suppose that $M_1$ is area enlargeable and that the universal covering of $M_2$ is spin. Then:
\begin{enumerate}
\item[(i)]\leavevmode\phantomsection\label{cor:main-i}$M_1\#M_2$ admits no complete metric of positive scalar curvature;
\item[(ii)]\leavevmode\phantomsection\label{cor:main-ii}every complete metric $g$ on $M_1\#M_2$ satisfies $\Lambda_\gamma(g)\le0$ for every $\gamma>\frac{m-1}{4m}$, and also for $\gamma=\frac{m-1}{4m}$ if $\scal_g$ is bounded from below.
\end{enumerate}
\end{maincor}

For $M_1=\T^m$ this is a spectral form of the generalized Geroch conjecture, valid in all dimensions for summands whose universal covering is spin. For $3\le m\le7$ and $\gamma>\frac{m-1}{4m}$, the conclusion of \hyperref[cor:main-ii]{(ii)} for $M_1=\T^m$ also follows from \cite[Theorem~1.21]{CS26}, without any spin assumption. For closed $M_1$ and spin $M_2$, part \hyperref[cor:main-i]{(i)} recovers Theorem~\ref{thm:WZ}. Further examples, with $M_1=W\times\mathbf R$ or $M_1$ a complete hyperbolic manifold of finite volume, are given in Corollaries~\ref{cor:product} and~\ref{cor:hyperbolic}.

The endpoint assertion in Corollary~\hyperref[cor:main-ii]{\ref*{cor:main}(ii)} follows from the next theorem, which treats the threshold $\gamma_m=\frac{m-1}{4m}$.

\begin{mainthm}\label{thm:D}
Let $m\ge2$ and $\gamma_m=\frac{m-1}{4m}$.
\begin{enumerate}
\item[(i)]\leavevmode\phantomsection\label{thm:D-i}Let $M$ be an area enlargeable $m$-manifold. Every complete metric $g$ on $M$ whose scalar curvature is bounded from below satisfies $\Lambda_{\gamma_m}(g)\le0$.
\item[(ii)]\leavevmode\phantomsection\label{thm:D-ii}For $a>0$ and a flat metric $g_{\T^{m-1}}$ on $\T^{m-1}$, the metric $g_a=dt^2+e^{2at^2}g_{\T^{m-1}}$ on the area enlargeable manifold $\T^{m-1}\times\mathbf R$ is complete, and
\[
\Lambda_\gamma(g_a)=
\begin{cases}
(m-1)\,a\,(1-4\gamma)+2a\sqrt{(m-1)\Bigl(\dfrac{m-1}{4}-m\gamma\Bigr)}, & \gamma\le\gamma_m,\\[8pt]
-\infty, & \gamma>\gamma_m.
\end{cases}
\]
In particular, $\Lambda_\gamma(g_a)>0$ for every $\gamma\le\gamma_m$, and $\Lambda_{\gamma_m}(g_a)=\frac{m-1}{m}\,a$.
\end{enumerate}
\end{mainthm}

Thus the range $\gamma>\gamma_m$ in Theorem~\ref{thm:A} is optimal,
and the endpoint obstruction does not hold in general without an additional
hypothesis, such as a lower bound on the scalar curvature.
For $\gamma<\gamma_m$ a classical example already shows that Theorem~\ref{thm:A} cannot be extended: the cusp metric $dt^2+e^{2t}g_{\T^{m-1}}$ on $\T^{m-1}\times\mathbf R$, a quotient of hyperbolic space, has $\scal\equiv-m(m-1)$ and $\lambda_1=\frac{(m-1)^2}{4}$, so that
\[
\Lambda_\gamma\bigl(dt^2+e^{2t}g_{\T^{m-1}}\bigr)=\frac{(m-1)(m-1-4m\gamma)}{4},
\]
which is positive for $\gamma<\gamma_m$ and vanishes at $\gamma=\gamma_m$, in accordance with part~\hyperref[thm:D-i]{(i)}. The metrics $g_a$ show that positivity can persist at the endpoint once the scalar curvature is unbounded from below.

These examples also bear on the rigidity theorem \cite[Theorem~1.20]{CS26}: for $3\le m\le7$, a complete metric on $\T^{m-1}\times\mathbf R$ admitting a positive function $u$ with $-\beta\Delta u+\frac12\scal\,u\ge0$ for some $0<\beta<\beta_m:=\frac{2m}{m-1}$ is flat. The range of $\beta$ there is sharp. Indeed, for a warped product $dt^2+e^{2h(t)}g_{\T^{m-1}}$ the function $u=e^{-(m-1)h/2}$ satisfies $-\beta_m\Delta u+\frac12\scal\,u=h''u$. For the cusp, $h(t)=t$, the right side vanishes, and for the metrics $g_a$, $h(t)=at^2$, it equals $2au>0$. Thus even strict positivity at $\beta=\beta_m$ does not force flatness. The same metrics show that the Dirac estimate behind Theorem~\ref{thm:A} breaks down at $\gamma_m$ (see Remark~\ref{rem:endpoint-dirac}).

\subsection*{Relation to earlier work}
Proposition~\ref{prop:spectral-sphere} gives a spectral estimate for sphere maps with the same threshold $\gamma>\frac{m-1}{4m}$ as the author's relative $\widehat A$-cowaist estimate \cite[Remark~1.5]{Liu26+}.
We include a self-contained proof using compactly supported test functions throughout, so that the argument applies directly to \eqref{eq:spectral-definition} without a lower bound on the scalar curvature. In odd dimensions, the proof uses two unitary maps that agree outside a compact set, together with Shi's relative spectral flow formula \cite[Theorem~2.12]{Shi25}. See also the relative spectral flow setup in \cite[\S2.2]{Liu26+} and the constructions in \cite[\S\S2.3 and 3]{Liu26L}.

The dual coefficient $\frac{4m}{m-1}$ is sharp in the quantitative Llarull theorem of \cite{Liu26L}, as shown by warped products over flat tori \cite[Proposition~1.2]{Liu26L}. Connected sums with a compact spin area enlargeable manifold with boundary are treated, under a mean-convex boundary condition, in \cite[Corollary~1.4]{Liu26L}. Theorem~\ref{thm:A} follows from Proposition~\ref{prop:spectral-sphere} and covering monotonicity of the spectral constant defined using compactly supported test functions. The main new contributions are Theorem~\ref{thm:B}, the normalization Lemma~\ref{lem:normalize-infinity}, and the endpoint analysis in Theorem~\ref{thm:D}.

Lott obtained related obstructions to positive scalar curvature using relative almost-flat classes \cite[Theorem~3.17 and Examples~3.18--3.19]{Lot25}. His application to connected sums in Example~3.18 treats a compact spin summand and a possibly noncompact spin summand. Theorem~\ref{thm:B} allows both summands to be noncompact and gives a criterion in terms of spin coverings for area enlargeability of their connected sum.

\subsection*{Plan of the paper}
Section~\ref{sec:prelim} collects the definitions and preliminary lemmas, including the normalization of maps at infinity. Section~\ref{sec:A} proves Theorem~\ref{thm:A}. Section~\ref{sec:B} proves Theorem~\ref{thm:B} and Corollary~\ref{cor:main}, and discusses examples. Section~\ref{sec:C} proves Theorem~\ref{thm:D}.

\subsection*{Notation}
Throughout the paper, manifolds are smooth, connected, oriented, second countable and without boundary, unless stated otherwise. Covering spaces need not be connected. They carry the pullback orientation, and Riemannian coverings carry the pullback metric.

On a spin manifold $N$, we write $S_N$ for the complex spinor bundle and $c$ for Clifford multiplication, with $c(\xi)^2=-|\xi|^2\operatorname{id}$. Tangent vectors and covectors are identified using the metric. In even dimensions, $S_N=S_N^+\oplus S_N^-$ carries its usual grading.

We write $\Ahat(N)$ for the $\widehat A$-form associated with the Levi--Civita connection of the metric in use. Integrals of characteristic forms are understood in top degree.

\section{Preliminaries}\label{sec:prelim}

\subsection{Area enlargeability}

Let $\Sph^m$ carry the round metric of curvature one. For a smooth map $f:M\to M'$ between Riemannian manifolds we write $\|\Lambda^2df_x\|$ for the operator norm of the induced map on two-vectors $\Lambda^2df_x:\Lambda^2T_xM\to\Lambda^2T_{f(x)}M'$, with respect to the norms induced by the metrics. We call this norm the \emph{area contraction constant} of $f$ at $x$. If $\sigma_1\ge\sigma_2\ge\dots$ are the singular values of $df_x$, then $\|\Lambda^2df_x\|=\sigma_1\sigma_2$. Hence $\|\Lambda^2df_x\|\le\varepsilon$ if and only if $|df(e_1)\wedge df(e_2)|\le\varepsilon$ for all orthonormal $e_1,e_2\in T_xM$, if and only if $|f^*\alpha|\le\varepsilon|\alpha|$ at $x$ for every $2$-form $\alpha$, with either the comass norm or the norm induced by the metrics. This is the $(\varepsilon,\Lambda^2)$-contraction of \cite[Definition~6.1]{GL83}. We call $f$ \emph{area $\varepsilon$-decreasing} if $\|\Lambda^2df_x\|\le\varepsilon$ for every $x$.

\begin{definition}\label{def:enlargeable}
Let $M$ be an $m$-manifold, $m\ge2$. A Riemannian metric $g$ on $M$ is \emph{area enlargeable} if for every $\varepsilon>0$ there exist a Riemannian covering $\pi:(\widehat M,\widehat g)\to(M,g)$ with $\widehat M$ spin and a smooth map $f:\widehat M\to\Sph^m$ such that
\begin{enumerate}
\item[(i)]\leavevmode\phantomsection\label{def:enlargeable-i}$f$ is area $\varepsilon$-decreasing with respect to $\widehat g$;
\item[(ii)]\leavevmode\phantomsection\label{def:enlargeable-ii}$f$ is \emph{locally constant at infinity}, that is, locally constant outside a compact subset of $\widehat M$;
\item[(iii)]\leavevmode\phantomsection\label{def:enlargeable-iii}$f$ has nonzero degree in the sense of Remark~\ref{rem:degree}.
\end{enumerate}
The manifold $M$ is \emph{area enlargeable} if every Riemannian metric on $M$, complete or not, is area enlargeable.
\end{definition}

This is the condition of Su \cite[Definition~1.2]{Su26}. The locally constant condition at infinity also appears in Shi \cite[Definition~1.10]{Shi25}. The classical definition of Gromov and Lawson \cite[Definitions~6.4 and~7.1]{GL83} requires $f$ to be constant outside a compact set. By Lemma~\ref{lem:normalize-infinity} and Remark~\ref{rem:classical} the two definitions agree when $m\ge2$.

\begin{remark}\label{rem:degree}
Let $m\ge2$, let $N$ be an oriented $m$-manifold, possibly disconnected and noncompact, and let $\phi:N\to\Sph^m$ be smooth and locally constant outside a compact set $K\subset N$. Fix an $m$-form $\omega$ on $\Sph^m$ with $\int_{\Sph^m}\omega=1$.

\leavevmode\phantomsection\label{rem:degree-a}\emph{(a)} If $\beta$ is a form of positive degree on $\Sph^m$, then $\phi^*\beta=0$ on $N\setminus K$, because $d\phi=0$ there. In particular $\phi^*\omega$ has compact support, and we set $\deg(\phi)=\int_N\phi^*\omega$. Two choices of $\omega$ differ by $d\beta$ with $\deg\beta=m-1\ge1$, and $\int_N\phi^*d\beta=\int_Nd(\phi^*\beta)=0$ by \hyperref[rem:degree-a]{(a)} and Stokes' theorem. Thus $\deg(\phi)$ does not depend on $\omega$. Only the finitely many components of $N$ meeting $K$ contribute, and $\deg(\phi)$ is the sum of the degrees of the restrictions of $\phi$ to these components.

\leavevmode\phantomsection\label{rem:degree-b}\emph{(b)} Let $y$ be a regular value of $\phi$. Since $d\phi=0$ outside $K$, the fiber $\phi^{-1}(y)$ is a closed discrete subset of $K$, hence finite, say $\{x_1,\dots,x_k\}$. We claim that $\deg(\phi)=\sum_i\operatorname{sign}\det d\phi_{x_i}$. In particular $\deg(\phi)$ is an integer. By \hyperref[rem:degree-a]{(a)} we may assume that $N$ is connected. Choose a compact smooth domain $W$ with $K\subset\operatorname{int}W$ (Lemma~\ref{lem:compact-core}), and let $E$ be the set of values of $\phi$ on the components of $N\setminus W$. It is finite by Lemma~\ref{lem:compact-core}, and $y\notin E$ because $\phi^{-1}(y)\subset K$.

Let $O_1,\dots,O_k$ be pairwise disjoint neighborhoods of $x_1,\dots,x_k$ on which $\phi$ is a diffeomorphism onto its image. For every sufficiently small ball $D$ around $y$ we have $\phi^{-1}(D)\subset O_1\cup\dots\cup O_k$. Otherwise there are points $z_l\notin\bigcup_iO_i$ with $\phi(z_l)\to y$. Infinitely many of them cannot lie in the compact set $W$, since a limit point would lie in $\phi^{-1}(y)\setminus\bigcup_iO_i=\emptyset$. Infinitely many of them cannot lie outside $W$, since then $\phi(z_l)\in E$, a finite set not containing $y$. Choosing $D\subset\bigcap_i\phi(O_i)$ and $\omega$ supported in $D$, each set $O_i\cap\phi^{-1}(D)$ is mapped diffeomorphically onto $D$, and the claim follows as in \cite[\S5]{Mil65}.

\leavevmode\phantomsection\label{rem:degree-c}\emph{(c)} Let $\{\phi_t\}_{t\in[0,1]}$ be a smooth homotopy such that every $\phi_t$ is locally constant outside a fixed compact set $K$. Then $\deg(\phi_0)=\deg(\phi_1)$. Indeed, $\phi_1^*\omega-\phi_0^*\omega=d\tau$ with $\tau=\int_0^1\iota_{\partial_t}\Phi^*\omega\,dt$ and $\Phi(x,t)=\phi_t(x)$. At points $x\notin K$ the form $\Phi^*\omega$ vanishes on $\partial_t\wedge\Lambda^{m-1}T_xN$, because $m-1\ge1$ and $d\phi_t=0$ near $x$. Hence $\tau$ has compact support and $\int_Nd\tau=0$.
\end{remark}

\begin{lemma}\label{lem:compact-core}
Let $N$ be a connected manifold and $L\subset N$ compact.
There is a nonempty compact smooth domain $W\subset N$ with $L\subset\operatorname{int}W$.
For every nonempty compact smooth domain $W$, the open set $N\setminus W$ has finitely many components, and the closure of each of them meets $\partial W$.
If $g$ is a complete metric on $N$, then the closure of each component of $N\setminus W$ is complete with respect to the length metric induced by $g$.
\end{lemma}

\begin{proof}
Fix $x_0\in N$.
Cover $L$ by finitely many relatively compact coordinate balls, and join their centers to $x_0$ by finitely many paths.
The union of the closures of these balls, the images of the paths, and $\{x_0\}$ is a nonempty compact connected set $A$ containing $L$.
Let $\rho:N\to[0,\infty)$ be a smooth proper function, and choose a regular value $c>\max_A\rho$.
Since $A$ is connected, it lies in a single connected component $W$ of $\rho^{-1}([0,c])$.
This $W$ is a nonempty compact connected smooth domain.
Since $A\subset\{p\in N:\rho(p)<c\}$, we have $L\subset\operatorname{int}W$.

Now let $W$ be any nonempty compact smooth domain, and let $\Omega$ be a component of $N\setminus W$.
We have $\overline{\Omega}\cap\operatorname{int}W=\emptyset$.
If $\overline{\Omega}\cap\partial W=\emptyset$, then $\overline{\Omega}\subset N\setminus W$.
Since $\Omega$ is closed in $N\setminus W$, it would then be closed in $N$.
It is also open in $N$, so connectedness would give $\Omega=N$, contrary to $W\ne\emptyset$.
Thus $\overline{\Omega}$ meets $\partial W$.

For each connected component $\Sigma$ of $\partial W$, choose an exterior collar
\[
E_\Sigma\cong\Sigma\times(0,\varepsilon)
\subset N\setminus W.
\]
This collar is connected and therefore lies in exactly one component of $N\setminus W$.
If $\overline{\Omega}$ meets $\Sigma$, then $\Omega$ meets $E_\Sigma$, so $E_\Sigma\subset \Omega$.
Consequently, every component of $N\setminus W$ is adjacent to at least one component of $\partial W$, while each component of $\partial W$ is adjacent to exactly one component of $N\setminus W$.
The compact manifold $\partial W$ has finitely many components, so $N\setminus W$ also has finitely many components.

The collar description shows that $\overline{\Omega}$ is a smooth manifold with boundary, whose boundary consists of the components of $\partial W$ adjacent to $\Omega$.
Let $d_{\overline{\Omega}}$ be the length metric induced by $g|_{\overline{\Omega}}$, and let $\{x_k\}$ be a $d_{\overline{\Omega}}$-Cauchy sequence.
Since $d_g\le d_{\overline{\Omega}}$, this sequence is $d_g$-Cauchy.
Completeness of $g$ gives a limit $x\in N$.
As $\overline{\Omega}$ is closed in $N$, we have $x\in\overline{\Omega}$.

Choose an interior or boundary chart $\varphi$ around $x$ whose image is a ball or a half-ball.
After shrinking the chart, the metric is bounded above by a constant multiple of the Euclidean metric.
For all sufficiently large $k$, the straight segment joining $\varphi(x_k)$ to $\varphi(x)$ lies in the chart image.
Its inverse image gives a curve in $\overline{\Omega}$, and hence
\[
d_{\overline{\Omega}}(x_k,x)
\le C\,|\varphi(x_k)-\varphi(x)|
\longrightarrow 0
\]
for some constant $C>0$.
Thus $(\overline{\Omega},d_{\overline{\Omega}})$ is complete.
\end{proof}

\begin{lemma}\label{lem:normalize-infinity}
Let $(N,g)$ be a connected oriented Riemannian $m$-manifold, $m\ge2$, and let $f:N\to\Sph^m$ be smooth and locally constant outside a compact set.
For every $y_0\in\Sph^m$ there is a smooth map $F:N\to\Sph^m$, equal to $y_0$ outside a compact set and equal to $f$ on a neighborhood of $\supp(df)$, such that
\[
\deg(F)=\deg(f)
\qquad\text{and}\qquad
\|\Lambda^2dF_x\|\le\|\Lambda^2df_x\|
\quad\text{for every }x\in N.
\]
\end{lemma}

\begin{proof}
If $N$ is compact, take $F=f$.
The condition outside a compact set is then vacuous.

Suppose that $N$ is noncompact.
The set $K=\supp(df)$ is compact because $f$ is locally constant outside a compact set.
By Lemma~\ref{lem:compact-core}, choose a nonempty compact smooth domain $W$ with $K\subset\operatorname{int}W$.
Write
\[
\partial W=\bigsqcup_{i=1}^{q}\Sigma_i
\]
for the connected components of its boundary.

Since $\partial W$ is compact and disjoint from $K$, we may choose smooth collar embeddings
\[
c_i:\Sigma_i\times(-2,2)\longrightarrow N\setminus K
\]
with pairwise disjoint images, such that $c_i(p,0)=p$ and
\[
c_i\bigl(\Sigma_i\times(-2,0)\bigr)
\subset\operatorname{int}W,
\qquad
c_i\bigl(\Sigma_i\times(0,2)\bigr)
\subset N\setminus W.
\]
Each collar image is connected and lies in $N\setminus K$, where $df=0$.
Thus $f\circ c_i$ is equal to a single point $z_i\in\Sph^m$.

Choose a smooth path $\alpha_i:[0,1]\to\Sph^m$ from $z_i$ to $y_0$ that is constant near both endpoints.
Set
\[
\mathcal{N}_i=c_i\bigl(\Sigma_i\times[0,1]\bigr),
\qquad
\mathcal{N}=\bigcup_{i=1}^{q}\mathcal{N}_i,
\]
and define
\[
F(x)=
\begin{cases}
f(x), & x\in W,\\
\alpha_i(t), & x=c_i(p,t),\quad 0<t\le1,\\
y_0, & x\in N\setminus(W\cup \mathcal{N}).
\end{cases}
\]
Near $c_i(\Sigma_i\times\{0\})$, the definitions on both sides are identically $z_i$.
Near $c_i(\Sigma_i\times\{1\})$, the definitions on both sides are identically $y_0$.
Therefore $F$ is smooth.
It agrees with $f$ on the open neighborhood $\operatorname{int}W$ of $K$, and it equals $y_0$ outside the compact set $W\cup \mathcal{N}$.

On $W$, we have $dF=df$.
On each collar $\mathcal{N}_i$, we have $df=0$, and $F$ depends only on the collar coordinate $t$.
Hence
\[
\operatorname{rank}(dF_x)\le1,
\qquad
\Lambda^2dF_x=0
\quad\text{for }x\in \mathcal{N}_i.
\]
On $N\setminus(W\cup \mathcal{N})$, both $df$ and $dF$ vanish.
These three regions cover $N$, so the required pointwise estimate follows.

Finally, let $\omega$ be the volume form on $\Sph^m$ normalized by $\int_{\Sph^m}\omega=1$.
On $W$, the pullbacks $F^*\omega$ and $f^*\omega$ agree because $F=f$ there.
Outside $W$, we have $df=0$ and $\operatorname{rank}(dF)\le1<m$, so both pullbacks vanish.
Thus
\[
F^*\omega=f^*\omega
\quad\text{on }N.
\]
Both forms have compact support, and the degree formula gives
\[
\deg(F)
=\int_N F^*\omega
=\int_N f^*\omega
=\deg(f).
\]
\end{proof}

\begin{remark}\label{rem:classical}
Let $m\ge2$. If $f:\widehat M\to\Sph^m$ is as in Definition~\ref{def:enlargeable}, then some component of $\widehat M$ carries a restriction of $f$ of nonzero degree by Remark~\hyperref[rem:degree-a]{\ref*{rem:degree}(a)}, and this component is a connected spin Riemannian covering of $M$. Lemma~\ref{lem:normalize-infinity} replaces the restriction by a map that is constant outside a compact set, has the same degree and is still area $\varepsilon$-decreasing. Hence a Riemannian metric is area enlargeable in the sense of Definition~\ref{def:enlargeable} if and only if it is area enlargeable in the sense of \cite[Definition~7.1]{GL83}. Since both notions of area enlargeable manifolds refer to all Riemannian metrics \cite[Definition~6.4]{GL83}, the same holds for manifolds. In particular, all area enlargeable manifolds of \cite[\S6]{GL83} are area enlargeable in the sense of Definition~\ref{def:enlargeable}.

The normalization uses that $\Lambda^2df=0$ whenever $df$ has rank at most one. It does not control the first derivative, so it says nothing about the corresponding question for length enlargeability.
\end{remark}

\subsection{Coverings and spectral monotonicity}

\begin{lemma}\label{lem:finiteness}
Let $\pi:\widehat N\to N$ be a covering map of manifolds, $B\subset N$ a coordinate ball and $V\subset N$ an open set with $\overline V\subset B$ compact. Then $B$ is evenly covered, $\pi^{-1}(B)=\bigsqcup_\alpha B_\alpha$, and every compact subset of $\widehat N$ meets only finitely many of the compact sets $L_\alpha=\pi^{-1}(\overline V)\cap B_\alpha$. In particular, the family $\{L_\alpha\}$ is locally finite.
\end{lemma}

\begin{proof}
The ball $B$ is connected, simply connected and locally path connected, hence evenly covered, by the lifting criterion and the uniqueness of lifts \cite[Propositions~1.33 and~1.34]{Hat02}. The set $L=\pi^{-1}(\overline V)$ is closed, and the sets $L_\alpha$ are pairwise disjoint, open in $L$ and compact, each being homeomorphic to $\overline V$. If $K\subset\widehat N$ is compact, then $K\cap L$ is compact and covered by the pairwise disjoint open subsets $K\cap L_\alpha$, so only finitely many of these are nonempty.
\end{proof}

\begin{proposition}[cf.\ {\cite[Theorem~1.1]{BMP18}}]\label{prop:spectrum-covering}
Let $\pi:(\widehat N,\widehat g)\to(N,g)$ be a Riemannian covering. Then $\Lambda_\gamma(\widehat g)\ge\Lambda_\gamma(g)$ for every $\gamma\in\mathbf R$.
\end{proposition}

\begin{proof}
Let $0\ne\psi\in C_c^\infty(\widehat N)$ and put $u(x)=\bigl(\sum_{y\in\pi^{-1}(x)}\psi(y)^2\bigr)^{1/2}$. By Lemma~\ref{lem:finiteness}, every point of $N$ has a neighborhood $V$ with local inverses $\sigma_1,\dots,\sigma_k$ of $\pi$ such that $\supp(\psi)\cap\pi^{-1}(V)\subset\sigma_1(V)\cup\dots\cup\sigma_k(V)$. On $V$ we have $u=|\phi|$ for the smooth map $\phi=(\psi\circ\sigma_1,\dots,\psi\circ\sigma_k):V\to\mathbf R^k$. Thus $u$ is locally Lipschitz, $\supp(u)\subset\pi(\supp\psi)$ is compact, and $|\nabla u|\le|\nabla\phi|$ almost everywhere. Since $\pi$ is a local isometry and $\scal_{\widehat g}=\scal_g\circ\pi$,
\[
\int_Nu^2=\int_{\widehat N}\psi^2,\qquad\int_N|\nabla u|^2\le\int_{\widehat N}|\nabla\psi|^2,\qquad\int_N\scal_g\,u^2=\int_{\widehat N}\scal_{\widehat g}\,\psi^2 .
\]
As Lipschitz test functions are admissible in \eqref{eq:spectral-definition}, the Rayleigh quotient of $\psi$ is at least $\Lambda_\gamma(g)$. Taking the infimum over $\psi$ proves the claim.
\end{proof}

\subsection{A Dirac estimate}

Let $(N,g)$ be a spin $m$-manifold and $E\to N$ a Hermitian bundle with a unitary connection. The twisted Dirac operator $D^E$ on $S_N\otimes E$ satisfies the Lichnerowicz formula
\[
(D^E)^2=\nabla^*\nabla+\frac{\scal_g}{4}+\mathcal R^E,\qquad \mathcal R^E=\sum_{i<j}c(e_i)c(e_j)\otimes R^E(e_i,e_j).
\]
We write $\|R^E\|_\infty=\sup_{x\in N}\sup_{v_1\perp v_2,\,|v_1|=|v_2|=1}\|R^E_x(v_1,v_2)\|_{\mathrm{op}}$, where $v_1,v_2\in T_xN$ and the operator norm is taken on $E_x$. Each summand of $\mathcal R^E$ has operator norm at most $\|R^E\|_\infty$, so
\begin{equation}\label{eq:curvature-contraction}
\mathcal R^E\ge-\binom m2\|R^E\|_\infty .
\end{equation}

\begin{lemma}\label{lem:dirac-gap}
Let $(N,g)$ be a complete spin $m$-manifold, $m\ge2$, and $E\to N$ a Hermitian bundle with unitary connection such that $\mathcal R^E\ge-\kappa$ for a constant $\kappa\ge0$. Let $\gamma>\frac{m-1}{4m}$ and put
\[
C_{m,\gamma}=\frac{m-1}{m}+\frac{1}{m^2\bigl(4\gamma-\frac{m-1}{m}\bigr)} .
\]
Then every nonzero $s\in C_c^\infty(N,S_N\otimes E)$ satisfies
\begin{equation}\label{eq:dirac-gap}
C_{m,\gamma}\,\|D^Es\|_{L^2}^2\ge\Bigl(\frac{\Lambda_\gamma(g)}{4\gamma}-\kappa\Bigr)\|s\|_{L^2}^2 .
\end{equation}
When $\Lambda_\gamma(g)=-\infty$, the estimate is vacuous.
If $\Lambda_\gamma(g)>4\gamma\kappa$, then the closure of $D^E$ is self-adjoint and invertible.
\end{lemma}

\begin{proof}
Let $c:T^*N\otimes(S_N\otimes E)\to S_N\otimes E$ be Clifford multiplication, so that $cc^*=m$, and let $P=\nabla-\frac1m c^*D^E$ be the twistor operator. Then $P$ takes values in $\ker c$, and
\[
|\nabla s|^2=|Ps|^2+\frac1m|D^Es|^2.
\]
At a point where $s\ne0$, let $v$ be a unit vector and $w=s/|s|$. The orthogonal projection of $v^\flat\otimes w$ to $\operatorname{im}c^*$ is $\frac1m c^*c(v^\flat\otimes w)$, of norm $\frac1{\sqrt m}$, so its projection to $\ker c$ has norm $\sqrt{(m-1)/m}$. Since $Ps\in\ker c$, it follows that
\[
\left|\operatorname{Re}\langle Ps,v^\flat\otimes w\rangle\right|
\le \sqrt{\frac{m-1}{m}}\,|Ps|.
\]
Moreover, by the adjoint relation and the fact that Clifford multiplication by a unit covector is an isometry,
\[
\begin{aligned}
\left|
\operatorname{Re}\left\langle
\frac1m c^*D^Es,v^\flat\otimes w
\right\rangle
\right|
&=\frac1m
\left|
\operatorname{Re}\langle D^Es,c(v^\flat\otimes w)\rangle
\right|\\
&\le \frac1m|D^Es|.
\end{aligned}
\]
Using
\[
d|s|(v)=\operatorname{Re}\langle\nabla s,v^\flat\otimes w\rangle
\]
and taking the supremum over unit vectors $v$, we obtain the inequality of Kato type
\[
\bigl|d|s|\bigr|
\le \sqrt{\frac{m-1}{m}}\,|Ps|+\frac1m|D^Es|.
\]
It holds almost everywhere, since $d|s|=0$ almost everywhere on the zero set of $s$. Put $\nu=\frac{4\gamma m}{m-1}-1$, which is positive because $\gamma>\frac{m-1}{4m}$. Then $(1+\nu)(m-1)/m=4\gamma$, and Young's inequality gives
\[
\begin{aligned}
\frac1{4\gamma}\bigl|d|s|\bigr|^2
&\le
\frac{(1+\nu)\frac{m-1}{m}|Ps|^2
+(1+\nu^{-1})\frac1{m^2}|D^Es|^2}{4\gamma}\\
&=|Ps|^2+
\frac{|D^Es|^2}
{m^2\bigl(4\gamma-\frac{m-1}{m}\bigr)}.
\end{aligned}
\]
Integrating and inserting
$\|Ps\|^2=\|\nabla s\|^2-\frac1m\|D^Es\|^2$
and the integrated Lichnerowicz formula, we get
\[
\begin{aligned}
C_{m,\gamma}\|D^Es\|^2
&\ge
\frac1{4\gamma}\int_N
\Bigl(\bigl|d|s|\bigr|^2+\gamma\scal_g|s|^2\Bigr)\,dV_g\\
&\quad+\int_N\langle\mathcal R^Es,s\rangle\,dV_g.
\end{aligned}
\]
Since $|s|$ is a compactly supported Lipschitz function,
\eqref{eq:spectral-definition} and $\mathcal R^E\ge-\kappa$
give \eqref{eq:dirac-gap}.

On a complete manifold $D^E$ is essentially self-adjoint on
$C_c^\infty$ \cite[Theorem~1.17]{GL83}, and
\eqref{eq:dirac-gap} extends to the domain of its closure.
If $\Lambda_\gamma(g)>4\gamma\kappa$, the closure is therefore
injective with closed range, and its range is dense because its
orthogonal complement is the kernel of the self-adjoint closure.
Hence it is invertible.
\end{proof}

\section{A spectral obstruction to area enlargeability}\label{sec:A}

For sphere maps, we record a version of the spectral estimate in \cite[Remark~1.5]{Liu26+}, using the formulation \eqref{eq:spectral-definition} with compactly supported test functions. For completeness, we give a direct proof in both parities and then apply it to prove Theorem~\ref{thm:A}.

\begin{proposition}\label{prop:spectral-sphere}
For every $m\ge2$ there is a constant $A_m>0$ with the following property. Let $(N,g)$ be a complete spin $m$-manifold, possibly disconnected, and let $f:N\to\Sph^m$ be a smooth map which is locally constant at infinity, has nonzero degree and satisfies $\|\Lambda^2df\|\le\eta$ everywhere. Then
\[
\Lambda_\gamma(g)\le A_m\,\gamma\,\eta\qquad\text{for every }\gamma>\frac{m-1}{4m},
\]
where $\Lambda_\gamma(g)$ is the infimum \eqref{eq:spectral-definition} over $C_c^\infty(N)$.
\end{proposition}

\begin{proof}
If $N$ is disconnected, then $\Lambda_\gamma(g)$ is the infimum of the spectral constants of the components, and by Remark~\hyperref[rem:degree-a]{\ref*{rem:degree}(a)} some component carries a restriction of $f$ of nonzero degree. We may therefore assume that $N$ is connected. The constants below depend only on $m$ and on fixed choices of bundles or maps on $\Sph^m$ representing the Bott generator (see \cite{GL83,LSWZ24,Shi25}). In each parity we choose a constant $k_m>0$, depending only on these data, that bounds the curvature terms as described below. We show that the inequality
\begin{equation}\label{eq:sphere-gap-condition}
\Lambda_\gamma(g)>4\gamma\,k_m\,\eta
\end{equation}
leads to a contradiction. The proposition then holds with $A_m=4k_m$. By Lemma~\ref{lem:compact-core}, choose a compact smooth domain $W_0\subset N$ with $\supp(df)\subset\operatorname{int}W_0$, together with a two-sided collar of $\partial W_0$ on which $df=0$. Then $N\setminus W_0$ has finitely many components $\Omega_1,\dots,\Omega_r$, and $f$ is constant on each $\Omega_j$ and on the adjacent collars. If $N$ is compact, we take $W_0=N$. Then $r=0$, the cutoff functions $\chi$ below vanish, the cutting steps are void, and the arguments reduce to the index theorem and the spectral flow formula on the closed manifold $N$.

Under \eqref{eq:sphere-gap-condition}, set
\[
a=\left(\frac{\Lambda_\gamma(g)/(4\gamma)-k_m\eta}{C_{m,\gamma}}\right)^{1/2}>0.
\]
The Dirac estimates below use this same constant $a$. In particular, it is independent of the interpolation parameter $t$ in the case of odd dimension.

\emph{Case~1: Even dimensions.} Let $E_0\to\Sph^m$ be a Hermitian bundle with unitary connection and $F_0$ the trivial bundle of the same rank with the trivial connection, such that $\int_{\Sph^m}\bigl(\operatorname{ch}(E_0)-\operatorname{ch}(F_0)\bigr)\ne0$. Put $E=f^*E_0$ and $F=f^*F_0$. The area contraction bound and \eqref{eq:curvature-contraction} give a constant $k_m>0$ such that
\begin{equation}\label{eq:even-curvature}
\mathcal R^E\ge-k_m\eta,\qquad \mathcal R^F=0.
\end{equation}
Both bundles are trivial with trivial connections on $N\setminus\operatorname{int}W_0$, where $f$ is locally constant. More precisely, $f$ is constant on each component of the open set $N\setminus\supp(df)\supset N\setminus\operatorname{int}W_0$, and choosing a unitary identification of the fibers of $E_0$ and $F_0$ over that value gives a parallel unitary isomorphism $\mathfrak t:E\to F$ over a neighborhood of $N\setminus\operatorname{int}W_0$. Thus $(E,F)$ is a Gromov--Lawson pair with support $W_0$. As in \cite[Example~2.5]{CZ24}, form the graded bundle
\[
S=S_N\widehat\otimes(E\oplus F^{\mathrm{op}}),
\]
where $E$ and $F$ are regarded as purely even bundles, $F^{\mathrm{op}}$ denotes $F$ with the opposite grading, and $\widehat\otimes$ denotes the graded tensor product. The Clifford multiplication is $c\oplus c$. The isomorphism $\mathfrak t$ gives this bundle the relative Dirac structure of \cite[Example~2.5]{CZ24}, with an odd involution $\sigma$ defined outside $W_0$. The involution $\sigma$ is self-adjoint and parallel, and it anticommutes with Clifford multiplication. Let $\mathcal D$ be its Dirac operator. For an odd Fredholm operator $\mathcal B$ on this graded bundle, we use the convention $\operatorname{ind}(\mathcal B)=\operatorname{ind}(\mathcal B^+)$, where $\mathcal B^+$ maps the positive to the negative graded summand. By \eqref{eq:even-curvature}, \eqref{eq:sphere-gap-condition} and Lemma~\ref{lem:dirac-gap}, the constant $a$ chosen above satisfies $\|\mathcal Ds\|\ge a\|s\|$ for all $s$ in the domain of $\mathcal D$.

Let $\chi:N\to[0,1]$ be smooth, $\chi=0$ near $W_0$ and $\chi=1$ outside a compact set, and let $0<\delta<a$. Then $\mathcal B=\mathcal D+\delta\chi\sigma$ is a Callias operator with admissible potential $\delta\chi$ in the sense of \cite[Definition~3.1]{CZ24}, and it is invertible, since $\|\delta\chi\sigma\|\le\delta<a$. Hence $\operatorname{ind}(\mathcal B)=0$.

We now compute this index. Choose a compact smooth domain $W$ with $W_0\subset\operatorname{int}W$ such that $\chi=1$ on a neighborhood of $N\setminus\operatorname{int}W$. By Lemma~\ref{lem:compact-core}, $N\setminus\operatorname{int}W$ is a finite disjoint union of complete manifolds with compact boundary, even if $N$ has infinitely many ends. We cut $N$ along the compact hypersurface $\partial W$ and impose the boundary sign $\textbf{s}=+1$ on both copies of $\partial W$. By \cite[Theorem~3.6]{CZ24}, $\operatorname{ind}(\mathcal B)$ is the sum of the indices of the resulting boundary value problems on $W$ and on $N\setminus\operatorname{int}W$. On $N\setminus\operatorname{int}W$ the potential $\psi=\delta\chi$ equals $\delta$, so $\psi^2-|d\psi|=\delta^2$ there and $\textbf{s}\psi=\delta>0$ along the boundary. The second index therefore vanishes by \cite[Lemma~3.7]{CZ24}.

By \cite[Corollary~3.9]{CZ24}, the first index equals the relative index of $(E,F)$, the index of the spin Dirac operator $D_{V(E,F)}$ on the double $\mathsf dW=W\cup_{\partial W}(-W)$ twisted by the bundle $V(E,F)$ obtained by gluing $E$ and $F$ along $\mathfrak t$. We give $\mathsf dW$ a metric that agrees with $g$ on $W_0$ and is a product near $\partial W$. Outside $W_0$ the bundle $V(E,F)$ is flat, and $F$ is flat everywhere. Thus $\operatorname{ch}V(E,F)-\operatorname{rk}E$ is supported in $W_0$, where it equals $\operatorname{ch}(E)-\operatorname{ch}(F)$. Moreover $\int_{\mathsf dW}\Ahat(\mathsf dW)=0$, because $\mathsf dW$ has an orientation-reversing diffeomorphism interchanging its two halves. The Atiyah--Singer index theorem therefore gives, up to a sign fixed by the grading conventions,
\[
\begin{aligned}
\operatorname{ind}(D_{V(E,F)})
&=\pm\int_W\Ahat(N)\wedge\bigl(\operatorname{ch}(E)-\operatorname{ch}(F)\bigr)\\
&=\pm\int_N\Ahat(N)\wedge f^*\bigl(\operatorname{ch}(E_0)-\operatorname{ch}(F_0)\bigr)\\
&=\pm\deg(f)\int_{\Sph^m}\bigl(\operatorname{ch}(E_0)-\operatorname{ch}(F_0)\bigr)\ne0.
\end{aligned}
\]
For the last equality, write $\operatorname{ch}(E_0)-\operatorname{ch}(F_0)=\theta_m+d\beta$, where $\theta_m$ is its component of degree $m$ and $\beta$ is a form of positive degree. Such a decomposition exists because $H^k(\Sph^m)=0$ for $0<k<m$, and the component of degree $0$ vanishes. By Remark~\hyperref[rem:degree-a]{\ref*{rem:degree}(a)}, $f^*\beta$ has compact support, so the exact term integrates to zero against the closed form $\Ahat(N)$, while $\int_Nf^*\theta_m=\deg(f)\int_{\Sph^m}\theta_m$. Hence $\operatorname{ind}(\mathcal B)\ne0$, a contradiction.

\emph{Case~2: Odd dimensions.} Here $m\ge3$. Fix a smooth map $\rho_0:\Sph^m\to SU(l)$ with $\int_{\Sph^m}\operatorname{ch}(\rho_0)\ne0$, where $\operatorname{ch}(\rho_0)$ denotes the odd Chern character form, with the normalized convention of \cite[\S3.2]{Shi25}. A unitary Bott generator can be made special unitary by adding the inverse of its determinant as an extra diagonal entry, which does not change the degree-$m$ component for $m\ge3$. Put $\rho^+=\rho_0\circ f$. The connections $d+t(\rho^+)^{-1}d\rho^+$ on the trivial bundle $\underline{\mathbf C}^l$, $0\le t\le1$, have curvature
\[
-t(1-t)\bigl((\rho^+)^{-1}d\rho^+\bigr)^2
=-t(1-t)f^*\bigl((\rho_0^{-1}d\rho_0)^2\bigr).
\]
The area contraction bound and \eqref{eq:curvature-contraction} give a constant $k_m>0$ with
\begin{equation}\label{eq:odd-curvature}
\mathcal R^{d+t(\rho^+)^{-1}d\rho^+}\ge-k_m\eta\qquad(0\le t\le1).
\end{equation}
The component of $\operatorname{ch}(\rho_0)$ in degree one is a multiple of $\operatorname{tr}(\rho_0^{-1}d\rho_0)=d\log\det\rho_0=0$. The components of degrees $3,5,\dots,m-2$ are closed, hence exact because $H^k(\Sph^m)=0$ for $0<k<m$, with primitives of positive degree. As in the even case,
\begin{equation}\label{eq:odd-pairing}
\int_N\Ahat(N)\wedge\operatorname{ch}(\rho^+)=\deg(f)\int_{\Sph^m}\operatorname{ch}(\rho_0)\ne0.
\end{equation}

We construct a comparison map as follows. Let $\rho_j\in SU(l)$ be the constant value of $\rho^+$ on $\Omega_j$, and let $I_l$ denote the $l\times l$ identity matrix. Choose smooth paths in $SU(l)$ from $I_l$ to $\rho_j$, constant near their endpoints, and define $\rho^-:N\to SU(l)$ by setting $\rho^-=I_l$ on $W_0$ minus an interior collar of $\partial W_0$. On the interior collar of each component of $\partial W_0$ adjacent to $\Omega_j$, use the corresponding path in the collar coordinate. Finally, set $\rho^-=\rho^+$ on $N\setminus W_0$. The differential of $\rho^-$ has rank at most one, so $\bigl((\rho^-)^{-1}d\rho^-\bigr)^2=0$. Hence the connections $d+t(\rho^-)^{-1}d\rho^-$ are flat, and $\operatorname{ch}(\rho^-)=0$: its component of degree one vanishes because $\rho^-$ takes values in $SU(l)$, and its higher components contain the square of $(\rho^-)^{-1}d\rho^-$. Moreover $\rho^+=\rho^-$ outside $W_0$, where both are locally constant.

As in \cite[\S2.3]{Shi25}, let $E=F$ be the trivial line bundle and form
\[
S=S_N\otimes(E\oplus F)\otimes\mathbf C^l.
\]
This bundle is ungraded. Equip it with Clifford multiplication $c\oplus(-c)$ on the two summands and the involution $\sigma$ interchanging them, both acting trivially on the factor $\mathbf C^l$. The pair $(E,F)$ is a Gromov--Lawson pair for the identity map. Let $\rho=\rho^+\oplus\rho^-$ act on the factor $\mathbf C^l$ of the two summands. Let $D_t^\pm$ be the spin Dirac operators twisted by $d+t(\rho^\pm)^{-1}d\rho^\pm$, and let $\mathcal D_t=D_t^+\oplus(-D_t^-)$. By \eqref{eq:odd-curvature}, \eqref{eq:sphere-gap-condition} and Lemma~\ref{lem:dirac-gap}, the same constant $a$ satisfies $\|\mathcal D_ts\|\ge a\|s\|$ for all $t\in[0,1]$. Choose $\chi$ and $0<\delta<a$ as in the even case, with $\chi=0$ near $W_0$. Both $d\rho^\pm$ and $\rho^+-\rho^-$ are supported in $W_0$. Then $\mathcal B_t=\mathcal D_t+\delta\chi\sigma$ is invertible for every $t$. Moreover $\chi\sigma$ commutes with $\rho$, because $\sigma$ interchanges the two summands, $\chi$ vanishes near $W_0$ and $\rho^+=\rho^-$ outside $W_0$. Hence
\[
\mathcal B_t=(1-t)\mathcal B_0+t\rho^{-1}\mathcal B_0\rho .
\]
The commutator $[\mathcal D_0,\rho]=c(d\rho^+)\oplus\bigl(-c(d\rho^-)\bigr)$ is bounded and compactly supported, so \cite[Assumption~2.5]{Shi25} holds, and $\{\mathcal B_t\}_{t\in[0,1]}$ is a Riesz-continuous family
of self-adjoint Fredholm operators. Its spectral flow $\operatorname{sf}(\mathcal B_0,\rho)$ vanishes because every $\mathcal B_t$ is invertible.

On the other hand, the Callias potential $\delta\chi$ is positive outside a compact set, and $\rho^+=\rho^-$ is locally constant outside a compact set. Let $W$ be a compact smooth domain with $W_0\subset\operatorname{int}W$ and $\chi=1$ on a neighborhood of $N\setminus\operatorname{int}W$, so that the potential equals $\delta$ there. By \cite[Theorem~2.12]{Shi25}, $\operatorname{sf}(\mathcal B_0,\rho)$ equals the spectral flow of the untwisted spin Dirac operator on the double $W\cup_{\partial W}(-W)$ along the family associated with the unitary map $\widetilde\rho$ that equals $\rho^+$ on $W$ and $\rho^-$ on $-W$. Here the metric on the double is chosen to agree with $g$ on $W_0$ and to be a product near $\partial W$. By the spectral flow formula on closed manifolds \cite[Theorem~2.8]{Get93}, and since the odd Chern character forms vanish near $\partial W$,
\[
\operatorname{sf}(\mathcal B_0,\rho)=\pm\int_W\Ahat(N)\wedge\bigl(\operatorname{ch}(\rho^+)-\operatorname{ch}(\rho^-)\bigr)=\pm\int_N\Ahat(N)\wedge\operatorname{ch}(\rho^+)\ne0,
\]
by $\operatorname{ch}(\rho^-)=0$ and \eqref{eq:odd-pairing}. The sign depends only on orientation conventions. This contradicts $\operatorname{sf}(\mathcal B_0,\rho)=0$.

In both cases \eqref{eq:sphere-gap-condition} is impossible, which proves the proposition.
\end{proof}

\begin{proof}[Proof of Theorem~\ref{thm:A}]
Let $g$ be a complete metric on $M$ and suppose that $\Lambda_\gamma(g)>0$ for some $\gamma>\frac{m-1}{4m}$. Choose $0<\eta<\Lambda_\gamma(g)/(A_m\gamma)$. Since $g$ is area enlargeable, there are a Riemannian covering $\pi:(\widehat M,\widehat g)\to(M,g)$ with $\widehat M$ spin and a map $f:\widehat M\to\Sph^m$ as in Definition~\ref{def:enlargeable}, with $\|\Lambda^2df\|\le\eta$. The metric $\widehat g$ is complete. By Propositions~\ref{prop:spectrum-covering} and~\ref{prop:spectral-sphere},
\[
\Lambda_\gamma(g)\le\Lambda_\gamma(\widehat g)\le A_m\gamma\eta<\Lambda_\gamma(g),
\]
a contradiction.
\end{proof}

\section{Area enlargeability of connected sums}\label{sec:B}

Throughout this section $M_1$ and $M_2$ are oriented $m$-manifolds, $m\ge2$. Fix closed coordinate balls $B_1\subset M_1$ and $B_2\subset M_2$ with centers $p_1$ and $p_2$. The connected sum $\mathcal M=M_1\#M_2$ is obtained by gluing $M_1\setminus\operatorname{int}B_1$ and $M_2\setminus\operatorname{int}B_2$ along their boundary spheres by an orientation-reversing diffeomorphism $\varphi:\partial B_2\to\partial B_1$, the smooth structure near the neck being defined by collars. We regard $M_1\setminus\operatorname{int}B_1$ and $M_2\setminus\operatorname{int}B_2$ as closed subsets of $\mathcal M$ that meet along the neck $\partial B_1=\varphi(\partial B_2)$. We assume that $B_1$ is a round ball centered at $p_1$ in a chart of $M_1$, and fix in this chart concentric open balls $U\subset V\subset V'$ with $B_1\subset U$, $\overline U\subset V$ and $\overline V\subset V'$. Similarly, $B_2$ is a round ball in a chart of $M_2$, and $B_2'\supset B_2$ is a concentric open ball in that chart.

The first step collapses a neighborhood of the neck on the enlargeable side, following Wang and Zhang \cite{WZ22}, who credit the idea to Gromov and Lawson \cite{GL83}. We formulate it for an arbitrary covering of $M_1$, so that no compactness is needed.

\begin{lemma}\label{lem:collapse}
Let $g$ be a Riemannian metric on $\mathcal M$. There are a Riemannian metric $g_1$ on $M_1$ with $g_1=g$ on $M_1\setminus U$ and a constant $c>0$ with the following property. Let $\pi_1:\widehat{M}_1\to M_1$ be a connected covering, $\varepsilon>0$, and $f:\widehat{M}_1\to\Sph^m$ a smooth map of nonzero degree which is locally constant at infinity and area $\varepsilon$-decreasing with respect to $\pi_1^*g_1$. Then for every $y_0\in\Sph^m$ there is a smooth map $F_1:\widehat{M}_1\to\Sph^m$ such that
\begin{enumerate}
\item[(i)]\leavevmode\phantomsection\label{lem:collapse-i}$F_1=y_0$ outside a compact set and on a neighborhood of $\pi_1^{-1}(B_1)$;
\item[(ii)]\leavevmode\phantomsection\label{lem:collapse-ii}$\deg(F_1)=\deg(f)$;
\item[(iii)]\leavevmode\phantomsection\label{lem:collapse-iii}$\|\Lambda^2dF_1\|\le c\varepsilon$ with respect to $\pi_1^*g_1$, and $\Lambda^2dF_1=0$ on $\pi_1^{-1}(U)$.
\end{enumerate}
\end{lemma}

\begin{proof}
Recall that $U\subset V\subset V'$ are the concentric coordinate balls centered at $p_1$ fixed above, with $B_1\subset U$, $\overline U\subset V$ and $\overline V\subset V'$.
Choose a metric $g_0$ on $M_1$ and a smooth function $\chi:M_1\to[0,1]$ with $\chi=1$ outside $U$ and $\chi=0$ on a neighborhood of $B_1$, and put $g_1=\chi g+(1-\chi)g_0$, where $\chi g$ is extended by zero. In the coordinates of the chart, centered at $p_1$, let $h(z)=\lambda(|z|)z$ on $V$ and $h=\mathrm{id}$ outside $V$, where $\lambda:[0,\infty)\to[0,1]$ is smooth, vanishes on a neighborhood of the radii of points of $\overline U$, and equals $1$ near and beyond the radius of $V$. Then $h$ is smooth, $h(U)=\{p_1\}$, and $H_t=(1-t)\,\mathrm{id}+t\,h$, computed in the chart on $V$ and equal to the identity outside $V$, is a smooth homotopy from $\mathrm{id}$ to $h$ with $H_t(V)\subset V$. Put $c=\sup_{M_1}\|\Lambda^2dh\|_{g_1}$, which is finite because $h=\mathrm{id}$ outside the compact set $\overline V$.

Let $\pi_1$, $f$ and $y_0$ be as in the statement. By Lemma~\ref{lem:normalize-infinity} we may assume that $f=y_0$ outside a compact set $K\subset \widehat{M}_1$. By Lemma~\ref{lem:finiteness}, the ball $V'$ is evenly covered. For $x\in\pi_1^{-1}(p_1)$, let $\sigma_x:V'\to \widehat{M}_1$ be the local inverse of $\pi_1$ with $\sigma_x(p_1)=x$, and put $V'_x=\sigma_x(V')$, $V_x=\sigma_x(V)$, $U_x=\sigma_x(U)$ and $B_{1,x}=\sigma_x(B_1)$. Define $\widehat H_t=\sigma_x\circ H_t\circ\pi_1$ on $V'_x$ and $\widehat H_t=\mathrm{id}$ on $\widehat{M}_1\setminus\pi_1^{-1}(\overline V)$. The two definitions agree on the overlap, where $H_t=\mathrm{id}$, so $\widehat H_t$ is a smooth homotopy. The map $\widehat h=\widehat H_1$ satisfies $\widehat h(U_x)=\{x\}$ and $\|\Lambda^2d\widehat h\|_{\pi_1^*g_1}\le c$, because $\pi_1$ and $\sigma_x$ are local isometries for $\pi_1^*g_1$ and $g_1$.

By Lemma~\ref{lem:finiteness}, only finitely many of the sets $\overline{V_x}$, $x\in\pi_1^{-1}(p_1)$, meet $K$. Hence
\[
K'=K\cup\bigcup_{\substack{x\in\pi_1^{-1}(p_1)\\\overline{V_x}\cap K\ne\emptyset}}\overline{V_x}
\]
is compact. We claim that $f\circ\widehat H_t=y_0$ on $\widehat{M}_1\setminus K'$ for every $t$. Let $z\notin K'$. If $z\notin\pi_1^{-1}(\overline V)$, then $\widehat H_t(z)=z\notin K$. If $z\in V_x$ with $\overline{V_x}\cap K=\emptyset$, then $\widehat H_t(z)\in V_x$, which is disjoint from $K$. If $z\in\partial V_x$ with $\overline{V_x}\cap K=\emptyset$, then $\widehat H_t(z)=z\notin K$. In each case $f(\widehat H_t(z))=y_0$.

By Remark~\hyperref[rem:degree-c]{\ref*{rem:degree}(c)}, the map $\bar f=f\circ\widehat h$ satisfies $\deg(\bar f)=\deg(f)$, and
\[
\|\Lambda^2d\bar f\|\le\|\Lambda^2df\|\,\|\Lambda^2d\widehat h\|\le c\varepsilon.
\]
Moreover $\bar f$ is constant on each $U_x$, with value $f(x)$, and $f(x)=y_0$ if $\overline{V_x}\cap K=\emptyset$.

It remains to make $\bar f$ equal to $y_0$ near the balls $B_{1,x}$ with $\overline{V_x}\cap K\ne\emptyset$. Let $r:V'\to[0,\infty)$ denote the Euclidean distance from $p_1$ in the fixed coordinates. For the radii $0<r_1<r_2$ of $B_1$ and $U$, respectively,
\[
B_1=\{p\in V':r(p)\le r_1\},\qquad
U=\{p\in V':r(p)<r_2\}.
\]
For each such $x$ choose a smooth path $\alpha_x:[r_1,r_2]\to\Sph^m$, equal to $y_0$ near $r_1$ and to $f(x)$ near $r_2$, and define $F_1=\alpha_x\circ r\circ\pi_1$ on $U_x\setminus B_{1,x}$ and $F_1=y_0$ on $B_{1,x}$. Elsewhere put $F_1=\bar f$. Then $F_1$ is smooth and satisfies \hyperref[lem:collapse-i]{(i)}. On $\pi_1^{-1}(U)$ the differential of $F_1$ has rank at most one, so $\Lambda^2dF_1=0$ there. Outside $\pi_1^{-1}(U)$ we have $F_1=\bar f$. This proves \hyperref[lem:collapse-iii]{(iii)}. Finally, $F_1^*\omega=\bar f^*\omega$, because both forms vanish on $\pi_1^{-1}(U)$ and the maps agree elsewhere. This gives \hyperref[lem:collapse-ii]{(ii)}.
\end{proof}

\begin{proof}[Proof of Theorem~\ref{thm:B}]
\emph{\hyperref[thm:B-ii]{(ii)}$\Leftrightarrow$\hyperref[thm:B-iii]{(iii)}.} If $M_2$ has a spin covering, then the universal covering factors through a component of it, and the pullback of a spin structure by a local diffeomorphism is a spin structure. The converse is trivial.

\emph{\hyperref[thm:B-i]{(i)}$\Rightarrow$\hyperref[thm:B-ii]{(ii)}.} Every oriented surface is spin, so let $m\ge3$. By \hyperref[thm:B-i]{(i)}, $\mathcal M$ has a spin covering $\pi:\widehat{\mathcal M}\to\mathcal M$. Let $Y=M_2\setminus\operatorname{int}B_2\subset\mathcal M$, a connected manifold with boundary sphere $\Sigma=\partial B_2$, and let $\widehat Y$ be a component of $\pi^{-1}(Y)$. Then $\pi:\widehat Y\to Y$ is a covering of manifolds with boundary, and every component of $\partial\widehat Y$ is mapped diffeomorphically onto $\Sigma$, because $\Sigma$ is simply connected. Glue a copy of $B_2$ to $\widehat Y$ along each component of $\partial\widehat Y$, via this diffeomorphism. The result is a manifold $\widehat{M}_2$, and the map to $M_2=Y\cup_\Sigma B_2$ given by $\pi$ on $\widehat Y$ and by the identity on each cap is a covering. Indeed, the preimage of a collar of $\Sigma$ in $Y$ is a disjoint union of collars, each meeting exactly one boundary component of $\widehat Y$, and each of them fits together with the corresponding cap to a sheet over the ball $B_2'$.

Since $m\ge3$, the collar of each boundary sphere is simply connected, so the restriction to it of the spin structure of $\widehat Y$ is isomorphic to the restriction of the spin structure of the cap together with the collar, an open ball. Gluing along such isomorphisms over the pairwise disjoint collars gives a spin structure on $\widehat{M}_2$.

\emph{\hyperref[thm:B-ii]{(ii)}$\Rightarrow$\hyperref[thm:B-i]{(i)}.} Let $\pi_2:\widehat{M}_2\to M_2$ be a spin covering. Passing to a component, we may assume that $\widehat{M}_2$ is connected. Let $g$ be a metric on $\mathcal M$, let $g_1$ and $c$ be as in Lemma~\ref{lem:collapse}, let $\eta>0$, and put $\varepsilon=\eta/c$. Since $M_1$ is area enlargeable, there are a spin Riemannian covering of $(M_1,g_1)$ and a map $f$ to $\Sph^m$ as in Definition~\ref{def:enlargeable}. By Remark~\hyperref[rem:degree-a]{\ref*{rem:degree}(a)}, we may restrict to a component $\widehat{M}_1$ on which $f$ has nonzero degree. Let $\pi_1:\widehat{M}_1\to M_1$ be the restricted covering. Fix $y_0\in\Sph^m$ and let $F_1:\widehat{M}_1\to\Sph^m$ be the map provided by Lemma~\ref{lem:collapse}.

As in the proof of Lemma~\ref{lem:collapse}, for $x\in\pi_1^{-1}(p_1)$ let $\sigma_x:V'\to\widehat{M}_1$ be the local inverse of $\pi_1$ with $\sigma_x(p_1)=x$, and put $B_{1,x}=\sigma_x(B_1)$. For $y\in\pi_2^{-1}(p_2)$, let $B_{2,y}\subset\widehat{M}_2$ be the lift of $B_2$ containing $y$ inside the sheet over $B_2'$. By Lemma~\ref{lem:finiteness}, applied to $V\subset V'$ and to an open ball $V_2$ with $B_2\subset V_2$ and $\overline{V_2}\subset B_2'$, both families are locally finite and consist of pairwise disjoint balls. The fibers $\pi_1^{-1}(p_1)$ and $\pi_2^{-1}(p_2)$ are countable, because they are closed discrete subsets of $\widehat{M}_1$ and $\widehat{M}_2$, respectively, and connected coverings of manifolds are second countable.

For every $y\in\pi_2^{-1}(p_2)$ take a copy $\widehat{M}_1^{\,y}$ of $\widehat{M}_1$, and for every $x\in\pi_1^{-1}(p_1)$ a copy $\widehat{M}_2^{\,x}$ of $\widehat{M}_2$. Let $\widehat{\mathcal M}$ be the manifold obtained from the disjoint union of the punctured manifolds $\widehat{M}_1^{\,y}\setminus\bigcup_x\operatorname{int}B_{1,x}$ and $\widehat{M}_2^{\,x}\setminus\bigcup_y\operatorname{int}B_{2,y}$ by gluing, for every pair $(x,y)\in\pi_1^{-1}(p_1)\times\pi_2^{-1}(p_2)$, the sphere $\partial B_{2,y}$ of the piece $\widehat{M}_2^{\,x}$ to the sphere $\partial B_{1,x}$ of the piece $\widehat{M}_1^{\,y}$ by the lift $\sigma_x\circ\varphi\circ\pi_2$ of $\varphi$, with the smooth structure near the gluing spheres given by the lifted collars. The coverings $\pi_1$ and $\pi_2$ combine to a map $\pi:\widehat{\mathcal M}\to\mathcal M$.

Since the families of balls are locally finite and the identifications are lifts of a single gluing, $\widehat{\mathcal M}$ is a Hausdorff smooth manifold, second countable because there are countably many pieces, and $\pi$ is a local diffeomorphism. It is a covering map. Away from the neck this follows from the corresponding property of $\pi_1$ and $\pi_2$. The preimage of the collar of the neck is the disjoint union, over all such pairs $(x,y)$, of the glued collars, each mapped diffeomorphically onto the collar in $\mathcal M$. Indeed, in $\widehat{M}_1^{\,y}$ the preimage of the $M_1$ half of the collar consists of its lifts near the spheres $\partial B_{1,x}$, $x\in\pi_1^{-1}(p_1)$, and in $\widehat{M}_2^{\,x}$ the preimage of the $M_2$ half consists of its lifts near $\partial B_{2,y}$, $y\in\pi_2^{-1}(p_2)$. These are glued in pairs with the same label $(x,y)$.

Each punctured piece is connected. Indeed, let $z_0,z_1$ be points of $\widehat{M}_1^{\,y}\setminus\bigcup_x\operatorname{int}B_{1,x}$ and join them by a path in $\widehat{M}_1^{\,y}$. Its image is compact, so it meets only finitely many of the balls $B_{1,x}$. For each of them in turn, replace the portion of the path between its first and its last point in $B_{1,x}$ by a path in the sphere $\partial B_{1,x}$, which is connected because $m\ge2$ and disjoint from the other balls. The same argument applies to the pieces $\widehat{M}_2^{\,x}$. Since every piece $\widehat{M}_1^{\,y}$ is glued to every piece $\widehat{M}_2^{\,x}$, the manifold $\widehat{\mathcal M}$ is connected.

The orientations of $\widehat{M}_1$ and $\widehat{M}_2$ induce an orientation of $\widehat{\mathcal M}$, because $\varphi$ reverses orientation, and $\pi$ preserves orientations. If $m=2$, then $\widehat{\mathcal M}$ is spin as an oriented surface. If $m\ge3$, each gluing sphere is simply connected, so near it the spin structures of the two adjacent pieces, pulled back from $\widehat{M}_1$ and $\widehat{M}_2$, agree up to isomorphism. As the gluing spheres are pairwise disjoint, these spin structures glue to a spin structure on $\widehat{\mathcal M}$, exactly as for a connected sum of two spin manifolds.

Fix $y_*\in\pi_2^{-1}(p_2)$. Define $F:\widehat{\mathcal M}\to\Sph^m$ by $F=F_1$ on $\widehat{M}_1^{\,y_*}\setminus\bigcup_x\operatorname{int}B_{1,x}$ and $F=y_0$ on all other pieces. By Lemma~\hyperref[lem:collapse-i]{\ref*{lem:collapse}(i)}, $F$ is smooth and equal to $y_0$ outside a compact set. Indeed, choose a compact set $K_F\subset \widehat{M}_1^{\,y_*}$ outside which the copy of $F_1$ equals $y_0$. Then $F=y_0$ off the image of the compact set $K_F\cap\bigl(\widehat{M}_1^{\,y_*}\setminus\bigcup_x\operatorname{int}B_{1,x}\bigr)$. The form $F^*\omega$ is supported in the piece $\widehat{M}_1^{\,y_*}\setminus\bigcup_x\operatorname{int}B_{1,x}$, where it equals $F_1^*\omega$, and $F_1^*\omega$ vanishes on each ball $B_{1,x}$. As $\pi$ preserves orientations, $\deg(F)=\deg(F_1)=\deg(f)\ne0$. For $x\in\pi_1^{-1}(p_1)$, Figure~\ref{fig:covering-construction} illustrates the gluing and the definition of $F$ near the neck $(x,y_*)$.

\begin{figure}[htbp]
\centering
\definecolor{coverblue}{HTML}{E8EFF4}
\begin{tikzpicture}[x=1cm,y=1cm,line cap=round,line join=round,
  every node/.style={font=\small,inner sep=2pt}]
\def\neckoutline{(-3.85,0.98)
  .. controls (-3.3,1.25) and (-2.91,0.82) .. (-2.22,0.98)
  .. controls (-1.35,1.16) and (-1.14,0.29) .. (0,0.29)
  .. controls (1.1,0.29) and (1.32,1.08) .. (2.02,1.02)
  .. controls (2.62,0.97) and (3.13,1.27) .. (3.85,1.1)
  --(3.85,-0.97)
  .. controls (3.12,-0.78) and (2.8,-1.16) .. (2.12,-1.01)
  .. controls (1.35,-0.83) and (1.08,-0.29) .. (0,-0.29)
  .. controls (-1.12,-0.29) and (-1.4,-1.08) .. (-2.14,-1.04)
  .. controls (-2.77,-1.01) and (-3.3,-0.76) .. (-3.85,-0.96)
  --cycle}
\path[fill=white] \neckoutline;
\begin{scope}
\clip (-5,-2) rectangle (0,2);
\path[fill=coverblue] \neckoutline;
\end{scope}
\draw[black,line width=.6pt] (-3.85,0.98)
  .. controls (-3.3,1.25) and (-2.91,0.82) .. (-2.22,0.98)
  .. controls (-1.35,1.16) and (-1.14,0.29) .. (0,0.29)
  .. controls (1.1,0.29) and (1.32,1.08) .. (2.02,1.02)
  .. controls (2.62,0.97) and (3.13,1.27) .. (3.85,1.1);
\draw[black,line width=.6pt] (3.85,-0.97)
  .. controls (3.12,-0.78) and (2.8,-1.16) .. (2.12,-1.01)
  .. controls (1.35,-0.83) and (1.08,-0.29) .. (0,-0.29)
  .. controls (-1.12,-0.29) and (-1.4,-1.08) .. (-2.14,-1.04)
  .. controls (-2.77,-1.01) and (-3.3,-0.76) .. (-3.85,-0.96);
\draw[black!45,line width=.5pt,dash pattern=on 2pt off 2pt]
  (0,.29) arc[start angle=90,end angle=270,x radius=.105cm,y radius=.29cm];
\draw[black!80,line width=.55pt]
  (0,-.29) arc[start angle=270,end angle=450,x radius=.105cm,y radius=.29cm];
\draw[black!65,line width=.4pt] (-.72,-.61)--(-.72,-.72)--(.72,-.72)--(.72,-.61);
\node at (-2.55,0) {$\widehat M_1^{\,y_*}$};
\node at (2.55,0) {$\widehat M_2^{\,x}$};
\node at (0,0.76) {$(x,y_*)$};
\node at (0,-1.01) {$F=y_0$};
\end{tikzpicture}
\caption{Local gluing of $\partial B_{1,x}$ and $\partial B_{2,y_*}$
at the neck $(x,y_*)$. The map $F$ equals $F_1$ on the shaded side
and $y_0$ on the other side, and is identically $y_0$ on the indicated collar.}
\label{fig:covering-construction}
\end{figure}

On the piece $\widehat{M}_1^{\,y_*}\setminus\bigcup_x\operatorname{int}B_{1,x}$ the metric $\pi^*g$ agrees with $\pi_1^*g_1$ outside $\pi_1^{-1}(U)$, and $\Lambda^2dF_1=0$ on $\pi_1^{-1}(U)$. On the other pieces $F$ is constant. Hence $\|\Lambda^2dF\|_{\pi^*g}\le c\varepsilon=\eta$. As $\eta>0$ was arbitrary, $g$ is area enlargeable, and as $g$ was arbitrary, so is $\mathcal M$.
\end{proof}

The maps $F$ constructed in the proof are constant outside a compact set. Thus, under the hypotheses of Theorem~\ref{thm:B}, the connected sum is area enlargeable in the classical sense of \cite{GL83} even when $M_1$ is only assumed to be area enlargeable in the locally constant sense.

\begin{proof}[Proof of Corollary~\ref{cor:main}]
By Theorem~\ref{thm:B}, $\mathcal M=M_1\#M_2$ is area enlargeable. Part \hyperref[cor:main-i]{(i)} follows from
Remark~\ref{rem:classical} and \cite[Theorem~6.12]{GL83}. Alternatively, one may apply \cite[Theorem~1.4]{Su26} directly. Part \hyperref[cor:main-ii]{(ii)} follows from Theorem~\ref{thm:A} and Theorem~\hyperref[thm:D-i]{\ref*{thm:D}(i)}.
\end{proof}

We record two families of examples with a noncompact enlargeable summand.

\begin{corollary}\label{cor:product}
Let $W$ be a closed enlargeable $(m-1)$-manifold in the sense of \cite[Definition~5.5]{GL83}, for instance $W=\T^{m-1}$, and let $X$ be an oriented $m$-manifold whose universal covering is spin. Then the conclusions of Corollary~\ref{cor:main} hold for $(W\times\mathbf R)\#X$.
\end{corollary}

\begin{proof}
By \cite[Proposition~6.7]{GL83} and Remark~\ref{rem:classical}, $W\times\mathbf R$ is area enlargeable. Apply Corollary~\ref{cor:main}.
\end{proof}

\begin{corollary}\label{cor:hyperbolic}
Let $H$ be an oriented complete hyperbolic $m$-manifold of finite volume, and let $X$ be an oriented $m$-manifold whose universal covering is spin. Then the conclusions of Corollary~\ref{cor:main} hold for $H\#X$.
\end{corollary}

\begin{proof}
By Corollary~\ref{cor:main} it suffices to show that $H$ is area enlargeable. This is shown in \cite[Example~6.11]{GL83}, and we recall the argument. A manifold that has an area enlargeable covering is itself area enlargeable, because every metric lifts to the covering and a composition of coverings of manifolds is a covering. If $H$ is closed, its universal covering is hyperbolic space, which is spin and admits maps of nonzero degree to the sphere that are constant outside a compact set and have arbitrarily small Lipschitz constant. Thus the hyperbolic metric on $H$ is enlargeable, and compactness transfers this property to every Riemannian metric on $H$ (see \cite[Definition~5.5]{GL83}). Hence $H$ is area enlargeable.

If $H$ is not closed, let $\Gamma_P\subset\pi_1(H)$ be the stabilizer of a parabolic fixed point corresponding to a cusp of $H$. By Bieberbach's theorem, $\Gamma_P$ contains a subgroup $\Lambda\cong\mathbf Z^{m-1}$ of finite index that acts by Euclidean translations on the horospheres centered at this point. In the upper half-space model with this point at infinity, the covering of $H$ corresponding to $\Lambda$ is $(\mathbf R^{m-1}/\Lambda)\times(0,\infty)\cong \T^{m-1}\times\mathbf R$. This manifold is area enlargeable by \cite[Proposition~6.7]{GL83} and Remark~\ref{rem:classical}.
\end{proof}

\section{Sharpness and the endpoint}\label{sec:C}

\begin{proof}[Proof of Theorem~\ref{thm:D}]
\emph{\hyperref[thm:D-i]{(i)}} Suppose that $\scal_g\ge-C$ for a constant $C\ge0$ and that $\lambda:=\Lambda_{\gamma_m}(g)>0$. For $\delta>0$ and $u\in C_c^\infty(M)$,
\[
\int_M\bigl(|\nabla u|^2+(\gamma_m+\delta)\scal_g\,u^2\bigr)\,dV_g\ge\int_M\bigl(|\nabla u|^2+\gamma_m\scal_g\,u^2\bigr)\,dV_g-\delta C\int_Mu^2\,dV_g ,
\]
so $\Lambda_{\gamma_m+\delta}(g)\ge\lambda-\delta C>0$ for small $\delta>0$. This contradicts Theorem~\ref{thm:A}.

\emph{\hyperref[thm:D-ii]{(ii)}} Write $n=m-1$, set $N=\T^n\times\mathbf R$, and let $h(t)=at^2$. The scalar curvature of the warped product $dt^2+e^{2h}g_{\T^n}$ over the flat torus is
\[
\scal_{g_a}=-2nh''-n(n+1)(h')^2=-4na-4n(n+1)a^2t^2 .
\]
The metric is complete, because $g_a\ge dt^2+g_{\T^n}$. Since $dV_{g_a}=e^{nh}\,dt\,dV_{g_{\T^n}}$, the substitution $w=e^{nh/2}u$ maps $C_c^\infty(N)$ bijectively onto itself. Integrating by parts in $t$, we obtain
\[
\int_N\bigl(|\nabla u|^2+\gamma\scal_{g_a}u^2\bigr)\,dV_{g_a}=\int_{\mathbf R\times \T^n}\Bigl(|\partial_tw|^2+e^{-2h}|\nabla_{\T^n}w|^2+V_\gamma\,w^2\Bigr)\,dt\,dV_{g_{\T^n}},
\]
where $\nabla_{\T^n}w(t,\cdot)$ denotes the gradient of $w(t,\cdot)$ with respect to the fixed flat metric $g_{\T^n}$, the norm in the torus term is taken with respect to the same metric, and
\[
V_\gamma=n\Bigl(\frac n4-(n+1)\gamma\Bigr)(h')^2+n\Bigl(\frac12-2\gamma\Bigr)h''=4na^2\Bigl(\frac n4-(n+1)\gamma\Bigr)t^2+na(1-4\gamma),
\]
and $\int_Nu^2\,dV_{g_a}=\int w^2\,dt\,dV_{g_{\T^n}}$. The term involving $\nabla_{\T^n}w$ is nonnegative, and it vanishes for functions $w$ independent of the torus variable. Hence $\Lambda_\gamma(g_a)$ equals the infimum $\mu_\gamma$ of the Rayleigh quotient of $-\partial_t^2+V_\gamma$ over $C_c^\infty(\mathbf R)$.
If $\mu_\gamma>-\infty$, the one-dimensional quadratic form bound,
applied to the slices $w(\cdot,x)$ and integrated over $x\in \T^n$,
gives $\Lambda_\gamma(g_a)\ge\mu_\gamma$.
Test functions independent of the torus variable give the reverse inequality
when $\mu_\gamma$ is finite and give
$\Lambda_\gamma(g_a)=-\infty$ when $\mu_\gamma=-\infty$.

Note that $\gamma_m=\frac n{4(n+1)}$. If $\gamma<\gamma_m$, then $V_\gamma=\omega^2t^2+na(1-4\gamma)$ with
\[
\omega=2a\sqrt{n\bigl(\frac n4-(n+1)\gamma\bigr)}>0.
\]
The harmonic oscillator $-\partial_t^2+\omega^2t^2$ has ground state energy $\omega$, with ground state $e^{-\omega t^2/2}$, and cutoffs of the ground state show that the infimum over $C_c^\infty(\mathbf R)$ is also $\omega$. Hence $\mu_\gamma=\omega+na(1-4\gamma)$.

If $\gamma=\gamma_m$, then $V_\gamma$ is the constant $na(1-4\gamma_m)=\frac{na}{n+1}$, and cutoffs supported on long intervals show that $\mu_{\gamma_m}=\frac{na}{n+1}=\frac{m-1}{m}a$.

If $\gamma>\gamma_m$, the coefficient of $t^2$ in $V_\gamma$ is negative, and translates of a fixed test function give $\mu_\gamma=-\infty$. Finally, the formula is positive for $\gamma\le\gamma_m$, because $\gamma_m<\frac14$.

The manifold $\T^{m-1}\times\mathbf R$ is area enlargeable by \cite[Proposition~6.7]{GL83} and Remark~\ref{rem:classical}.
\end{proof}

\begin{remark}\label{rem:endpoint-dirac}
Put $n=m-1$ and $h(t)=at^2$, and let $D$ be the untwisted spin Dirac operator of $g_a$.
The metrics $g_a$ also show that Lemma~\ref{lem:dirac-gap} fails at $\gamma=\gamma_m$: there is no constant $C>0$ such that $C\|Ds\|^2\ge\Lambda_{\gamma_m}(g)\|s\|^2$ for every complete spin manifold and every compactly supported spinor $s$. Indeed, equip $\T^{m-1}\times\mathbf R$ with the spin structure induced from the trivial spin structure on $\T^{m-1}$, for which constant spinors on $\T^{m-1}$ are parallel. On spinors that are constant along the torus, the Dirac operator of $g_a$ is $c(\partial_t)\bigl(\partial_t+\frac n2h'\bigr)$, and conjugation by $e^{nh/2}$ turns it into $c(\partial_t)\partial_t$. Fix $0\ne\phi\in C_c^\infty(\mathbf R)$ and a nonzero constant spinor $\psi_0$, and put $s_L=e^{-nh/2}L^{-1/2}\phi(t/L)\psi_0$. Then $\|s_L\|$ does not depend on $L$, while $\|Ds_L\|=O(L^{-1})$. Hence the Dirac operator of $g_a$ is not invertible, although $\Lambda_{\gamma_m}(g_a)>0$.
\end{remark}

\section*{Acknowledgments}
The author is grateful to Professors Weiping Zhang, Guangxiang Su, Zhenlei Zhang and Bo Liu for their encouragement and support. This work was partially supported by the National Natural Science Foundation of China (12501064), the China Postdoctoral Science Foundation (2025M773075, 2025T002TJ, GZC20252016) and the Nankai Zhide Foundation.


\begin{thebibliography}{BHHSZ26}

\bibitem[BMP18]{BMP18} W.~Ballmann, H.~Matthiesen and P.~Polymerakis, On the bottom of spectra under coverings, \emph{Math.\ Z.} \textbf{288} (2018), no.~3--4, 1029--1036.

\bibitem[BHHSZ26]{BHHSZ26} Y.~Bi, T.~Hao, S.~He, Y.~Shi and J.~Zhu, A proof for the Riemannian positive mass theorem up to dimension $19$, preprint, arXiv:2603.02769 (2026).

\bibitem[BZ26]{BZ26} Y.~Bi and J.~Zhu, Positive scalar curvature obstructions via singular dimension descent, preprint, arXiv:2606.20528 (2026).

\bibitem[CS21]{CS21} S.~Cecchini and T.~Schick, Enlargeable metrics on nonspin manifolds, \emph{Proc.\ Amer.\ Math.\ Soc.} \textbf{149} (2021), no.~5, 2199--2211.

\bibitem[CZ24]{CZ24} S.~Cecchini and R.~Zeidler, Scalar and mean curvature comparison via the Dirac operator, \emph{Geom.\ Topol.} \textbf{28} (2024), no.~3, 1167--1212.

\bibitem[CS26]{CS26} X.~Chai and Y.~Sun, Some rigidity theorems for spectral curvature bounds, preprint, arXiv:2604.04052 (2026).

\bibitem[Che24]{Che24} S.~Chen, A generalization of the Geroch conjecture with arbitrary ends, \emph{Math.\ Ann.} \textbf{389} (2024), no.~1, 489--513.

\bibitem[CL24]{CL24} O.~Chodosh and C.~Li, Generalized soap bubbles and the topology of manifolds with positive scalar curvature, \emph{Ann.\ of Math.\ (2)} \textbf{199} (2024), no.~2, 707--740.

\bibitem[CMSW25]{CMSW25} O.~Chodosh, C.~Mantoulidis, F.~Schulze and Z.~Wang, Generic regularity for minimizing hypersurfaces in dimension $11$, preprint, arXiv:2506.12852 (2025).

\bibitem[FCS80]{FCS80} D.~Fischer-Colbrie and R.~Schoen, The structure of complete stable minimal surfaces in $3$-manifolds of nonnegative scalar curvature, \emph{Comm.\ Pure Appl.\ Math.} \textbf{33} (1980), no.~2, 199--211.

\bibitem[Get93]{Get93} E.~Getzler, The odd Chern character in cyclic homology and spectral flow, \emph{Topology} \textbf{32} (1993), no.~3, 489--507.

\bibitem[GL80]{GL80} M.~Gromov and H.~B.~Lawson, Spin and scalar curvature in the presence of a fundamental group.~I, \emph{Ann.\ of Math.\ (2)} \textbf{111} (1980), no.~2, 209--230.

\bibitem[GL83]{GL83} M.~Gromov and H.~B.~Lawson, Positive scalar curvature and the Dirac operator on complete Riemannian manifolds, \emph{Inst.\ Hautes \'Etudes Sci.\ Publ.\ Math.} \textbf{58} (1983), 83--196.

\bibitem[Hat02]{Hat02} A.~Hatcher, \emph{Algebraic topology}, Cambridge University Press, Cambridge, 2002.

\bibitem[HSY25]{HSY25} S.~He, Y.~Shi and H.~Yu, Positive mass theorems on singular spaces and some applications, preprint, arXiv:2502.18000 (2025).

\bibitem[HKKZ24]{HKKZ24} S.~Hirsch, D.~Kazaras, M.~Khuri and Y.~Zhang, Spectral torical band inequalities and generalizations of the Schoen--Yau black hole existence theorem, \emph{Int.\ Math.\ Res.\ Not.\ IMRN} \textbf{2024} (2024), no.~4, 3139--3175.

\bibitem[LUY24]{LUY24} M.~Lesourd, R.~Unger and S.-T.~Yau, Positive scalar curvature on noncompact manifolds and the Liouville theorem, \emph{Comm.\ Anal.\ Geom.} \textbf{32} (2024), no.~5, 1311--1337.

\bibitem[LSWZ24]{LSWZ24} Y.~Li, G.~Su, X.~Wang and W.~Zhang, Llarull's theorem on odd dimensional manifolds: the noncompact case, \emph{J.\ Differential Geom.}, to appear; arXiv:2404.18153.

\bibitem[Liu26a]{Liu26L} D.~Liu, Sharp quantitative Llarull's theorem and scalar-curvature rigidity with boundary, preprint, arXiv:2601.15043 (2026).

\bibitem[Liu26b]{Liu26+} D.~Liu, A sharp inequality relating scalar curvature, bottom spectrum, and the relative $\widehat A$-cowaist on complete manifolds, preprint, arXiv:2603.20864 (2026).

\bibitem[Lot25]{Lot25} J.~Lott, Some obstructions to positive scalar curvature on a noncompact manifold, \emph{J.\ reine angew.\ Math.} \textbf{829} (2025), 247--283.

\bibitem[Mil65]{Mil65} J.~Milnor, \emph{Topology from the differentiable viewpoint}, University Press of Virginia, Charlottesville, 1965.

\bibitem[SY79]{SY79} R.~Schoen and S.-T.~Yau, On the structure of manifolds with positive scalar curvature, \emph{Manuscripta Math.} \textbf{28} (1979), no.~1--3, 159--183.

\bibitem[Shi25]{Shi25} P.~Shi, Spectral flow of Callias operators, odd K-cowaist, and positive scalar curvature, \emph{Adv.\ Math.} \textbf{479} (2025), Paper No.~110429.

\bibitem[Su26]{Su26} G.~Su, A remark on $\Lambda^2$-enlargeable manifolds, \emph{Sci.\ China Math.} (2026), \url{https://doi.org/10.1007/s11425-025-2590-7}.

\bibitem[WZ22]{WZ22} X.~Wang and W.~Zhang, On the generalized Geroch conjecture for complete spin manifolds, \emph{Chinese Ann.\ Math.\ Ser.~B} \textbf{43} (2022), no.~6, 1143--1146.

\bibitem[Zha20]{Zha20} W.~Zhang, Nonnegative scalar curvature and area decreasing maps, \emph{SIGMA Symmetry Integrability Geom.\ Methods Appl.} \textbf{16} (2020), Paper No.~033.

\end{thebibliography}
\end{document}